\documentclass[12pt]{amsart}
\usepackage{amsfonts}
\usepackage{amscd}
\usepackage{amsmath, mathrsfs, amssymb, mathtools}
\usepackage{amsthm}
\usepackage{color}
\usepackage{setspace}
\usepackage[colorlinks,linkcolor=blue,citecolor=purple]{hyperref}
\usepackage{epsfig}
\usepackage{here}
\usepackage{graphicx}
\usepackage[all]{xy}
\usepackage{psfrag}
\usepackage{graphicx,transparent}
\usepackage{enumerate}
\usepackage{caption}
\usepackage{tikz-cd}
\usepackage[textheight=8.75in, textwidth=6.75in]{geometry}
\usepackage[numbers,sort&compress]{natbib}
\theoremstyle{plain}
\newtheorem{theorem}{Theorem}[section]
\newtheorem{lemma}[theorem]{Lemma}

\newtheorem{conjecture}[theorem]{Conjecture}

\theoremstyle{definition}
\newtheorem{remark}{Remark}

\usepackage[obeyFinal,textsize=tiny,shadow,loadshadowlibrary]{todonotes}
\usepackage{soul}
\usepackage[normalem]{ulem}
\newcommand{\added}[1]{\ifmmode{\color{red}#1}\else{\color{red}\uline{#1}}\fi}

\begin{document}

\title{Parity of $k$-differentials in genus zero and one}

\date{\today}

\author{Dawei Chen}
\address{Department of Mathematics, Boston College, Chestnut Hill, MA 02467}
\email{dawei.chen@bc.edu}

\address{Axiom Math, 124 University Avenue, Palo Alto, CA 94301}

\author{Evan Chen}
\email{evan@axiommath.ai}

\author{Kenny Lau}
\email{kenny@axiommath.ai}

\author{Ken Ono}
\email{ken@axiommath.ai}

\author{Jujian Zhang}
\email{jujian@axiommath.ai}


\thanks{Research of D.C.\ is supported in part by the National Science Foundation
    under Grant DMS-2301030 and by Simons Travel Support for Mathematicians.}

\begin{abstract}
    Here we completely determine the spin parity of $k$-differentials with
    prescribed zero and pole orders on Riemann surfaces of genus zero and one.
    This result was previously obtained conditionally by the first author and
    Quentin Gendron assuming the truth of a number-theoretic hypothesis \cite[Conjecture A.10]{CG22}.
    We prove this hypothesis by reformulating it in terms of Jacobi symbols,
    reducing the proof to a combinatorial identity and standard facts about Jacobi symbols.
    The proof was obtained by AxiomProver and the system formalized the proof
    of the combinatorial identity in Lean/Mathlib (see the Appendix).
    We emphasize that it is this combinatorial identity, and not the
    geometric results on $k$-differentials, that was formalized. The formalized
    statement is displayed and explained in the Appendix.
\end{abstract}

\maketitle

\section{Introduction and Statement of Results}
\label{sec:intro}

A (meromorphic) $k$-differential $\xi$ is a section of the
$k$\textsuperscript{th} power of the canonical bundle on a Riemann surface $X$ of genus $g$.
For a tuple of integers $\mu = (m_1,\ldots, m_n)$ where
\[ m_1+\cdots +m_n = k(2g-2), \]
we let $\Omega^k\mathcal M_g(\mu)$ be the moduli space of (primitive) $k$-differentials
$(X,\xi)$ with zero and pole orders given by $\mu$.
A $k$-differential induces a flat metric with conical singularities on $X$,
where the cone angles are multiples of $2\pi / k$, determined by the orders at the zeros and poles.
From this perspective, the study of $k$-differentials of a prescribed
type plays an important role in surface dynamics, moduli theory, and combinatorial enumeration.
We refer the reader to \cite{Zo06, Wr15, Ch17, BCGGMk} for an introduction to this fascinating subject.

Although $\Omega^k\mathcal M_g(\mu)$ is a complex orbifold, it can be disconnected for special $\mu$.
The classification of connected components of $\Omega^k\mathcal M_g(\mu)$ has
been completed for holomorphic differentials ($k=1$ and $m_i \geq 0$ for all $i$, \cite{KZ03});
meromorphic differentials ($k=1$ and some $m_i < 0$, \cite{Bo15});
quadratic differentials of finite area ($k=2$ and $m_i \geq -1$ for all $i$, \cite{La04H, La04S, La08, CM14});
and quadratic differentials of infinite area ($k=2$ and some $m_i < -1$, \cite{CG22}).
However, for $k\geq 3$, the classification of connected components of
$\Omega^k\mathcal M_g(\mu)$ remains largely unknown.

In the above results, besides the (easy-to-understand) hyperelliptic structure
and certain ad hoc structures occurring in low genus (e.g., the rotation number in genus one),
the only other known invariant that can help distinguish connected components of
$\Omega^k\mathcal M_g(\mu)$ is the spin parity.

For $k = 1$, let $\omega$ be a differential one-form whose zero and pole orders are even.
Then the half-canonical divisor ${\rm div}(\omega)/2$ defines a theta-characteristic, whose spin parity
\[ \dim H^0\bigl(X, {\rm div}(\omega)/2\bigr) \pmod{2} \]
is deformation invariant (see~\cite{At71, Mu71}).
This parity coincides with the Arf invariant defined by the flat surface
structure of $(X, \omega)$ (see~\cite{Jo80} and \cite[Section 3]{KZ03}).

Given a $k$-differential $(X, \xi)$ parameterized by $\Omega^k\mathcal M_g(\mu)$ with $k > 1$,
there exists a canonical cyclic cover $\pi\colon \widehat{X}\to X$ of degree $k$
such that $\pi^{*}\xi = \widehat{\omega}^k$,
where $\widehat{\omega}$ is a differential one-form on $\widehat{X}$ whose zero
and pole orders are uniquely determined by $\mu$ (see~\cite[Section 2.1]{BCGGMk}).
If the zero and pole orders of $\widehat{\omega}$ are even,
we say that $\Omega^k\mathcal M_g(\mu)$ is of parity type.
Consequently, we use the spin parity of $(\widehat{X}, \widehat{\omega})$ to
define the spin parity of $(X, \xi)$ in $\Omega^k\mathcal M_g(\mu)$.

It is thus natural to ask how to determine the spin parity for $k$-differentials of parity type.
In contrast to the case $k=1$, when $k$ is even,
all $k$-differentials in $\Omega^k\mathcal M_g(\mu)$ of parity type have the same spin parity,
which can be determined explicitly from $\mu$ (see \cite[Theorem 1.2]{La04S} for
$k=2$ and \cite[Section 5.2]{CG22} for general even $k$). However, this is not the case when $k$ is odd.
Indeed, even if $X$ has low genus,
the domain $\widehat{X}$ of the canonical cyclic cover can have high genus,
which makes the computation of spin parity nontrivial.
Moreover, if $k$ is odd, then $\Omega^k\mathcal M_g(\mu)$ is of parity type if
and only if all entries of $\mu$ are even (see \cite[Proposition 5.1]{CG22}).

In this context, for odd $k$,
the spin parity of $k$-differentials in genus zero and one was studied explicitly in \cite[Appendix]{CG22}.
After a series of delicate geometric arguments and based on numerical evidence,
the entire study eventually reduced to a number-theoretic conjecture (see also
\cite[Conjecture A.8 and Remark A.9]{CG22} for another equivalent form of the conjecture).

Throughout, for an odd integer $k \geq 3$ and \emph{any} integer $n$,
we let $N_k(n)$ be the number of pairs of positive integers $(b_1, b_2)$
such that $b_1, b_2 \leq (k-1)/2$, $b_1 + b_2 \geq (k+1)/2$, and
$b_2 \equiv n b_1 \pmod{k}$. This definition places no coprimality
constraint on $n$. The hypothesis $\gcd(n,k) = \gcd(n+1,k) = 1$ is imposed only
where it is needed, as in the conjecture below.

\begin{conjecture}[{\cite[Conjecture A.10]{CG22}}]
\label{conj}
Let $k \geq 3$ be an odd integer, and let $n$ be an integer with
$\gcd (n, k) = \gcd (n+1, k) = 1$. Then we have
\[ N_k(n) \equiv \left\lfloor\frac{k+1}{4}\right\rfloor \pmod{2}. \]
\end{conjecture}

In this paper, we prove this conjecture.

\begin{theorem}\label{thm:conj} Conjecture~\ref{conj} is true.
\end{theorem}

Consequently, the results in \cite[Appendix]{CG22} that were originally stated
conditionally on Conjecture~\ref{conj} become unconditional,
thereby completely determining the spin parity of $k$-differentials in genus zero and one for all $k$.

In order to state these results, we recall the $q$-adic valuation $\nu_q(m)$
for a prime $q$ and an integer $m$,
which is given by the largest exponent $\nu$ such that $q^{\nu}$ divides $m$.
Additionally, for odd $k$, consider the prime factorization
$k = p_1^{h_1}\cdots p_s^{h_s}q_1^{\ell_1}\cdots q_t^{\ell_t}$,
where each $p_i$ is an odd prime such that $\lfloor (p_i+1)/4 \rfloor$ is even
and each $q_i$ is an odd prime such that $\lfloor (q_i+1)/4 \rfloor$ is odd.
For the collection of primes $\mathcal{Q} = \{ q_1, \ldots, q_t \}$ in the factorization of $k$,
define \[ \nu_{\mathcal{Q}}(m) \coloneq \sum_{i=1}^t \min\{\nu_{q_i}(m), \nu_{q_i}(k)\}. \]
Finally, for $\mu = (m_1, \ldots, m_n)$, we define
\begin{equation}
\label{eq:n_k}
    n_k(\mu) \coloneq \# \left \{ i :
    \nu_{\mathcal{Q}}(m_i) \not \equiv \nu_{\mathcal{Q}}(k) \pmod{2}\right \}
\end{equation}
(see \eqref{eq:n_k_Jacobi} for an equivalent definition using Jacobi symbols).
We also denote by $2\mu$ the tuple $(2m_1,\ldots, 2m_n)$.

Since we prove Conjecture~\ref{conj} here, the following two theorems from \cite{CG22} are now unconditional.

\begin{theorem}[{\cite[Theorem A.16]{CG22}}]
\label{thm:zero}
For genus zero and odd $k$, the spin parity of $\Omega^k\mathcal M_0(2\mu)$ is
determined by $n_k(\mu) \pmod{2}$.
\end{theorem}

\begin{theorem}[{\cite[Theorem A.21]{CG22}}]
\label{thm:one}
For genus one and odd $k$, the spin parity of the connected component
$\Omega^k\mathcal M_1(2\mu)^d$ of rotation number $d$ is determined by $n_k(\mu) + d + 1 \pmod{2}$.
\end{theorem}

The rest of the paper is structured as follows.
In Section~\ref{sec:proof} we prove Theorem~\ref{thm:conj}.
The key observation that had not been noticed earlier is that the parity
condition in Conjecture~\ref{conj} can be reformulated in terms of Jacobi symbols,
which are standard generalizations of Legendre symbols.
Combined with some elementary number-theoretic arguments, we obtain the proof of Theorem~\ref{thm:conj}.
This key observation, as well as the formal proof of the main
lemma, Lemma~\ref{lem:N_as_floor_diff}, was autonomously discovered by AxiomProver.
In the Appendix, we discuss the AI architecture and provide the Lean verification artifacts.
This section is modular and may be omitted by readers primarily interested in the number-theoretic results.

\subsection*{Acknowledgements}
This work was initiated during the $2026$ Joint Mathematics Meetings,
where the first author brought Conjecture~\ref{conj} to the attention of the fourth author.
The authors thank the conference organizers for creating a wonderful, interactive environment. The authors thank the referee for helpful comments.

This paper describes a test case for AxiomProver,
an autonomous system that is currently under development.
The project engineering team is
Chris Cummins,
GSM,
Dejan Grubisic,
Leopold Haller,
Letong Hong (principal investigator),
Andranik Kurghinyan,
Kenny Lau,
Hugh Leather,
Aram Markosyan,
Manooshree Patel,
Gaurang Pendharkar,
Vedant Rathi,
Alex Schneidman,
Volker Seeker,
Shubho Sengupta (principal investigator),
Ishan Sinha,
Jimmy Xin,
and Jujian Zhang.

\section{Proof of Conjecture~\ref{conj}}\label{sec:proof}

The proof of Conjecture~\ref{conj} is based on the observation that
$\lfloor\frac{k+1}{4}\rfloor\pmod 2$ is periodic modulo $8$,
and can be reformulated in terms of the Jacobi symbol $\Big(\frac{2}{k}\Big).$

\subsection{Preliminaries about Jacobi symbols}

We recall three standard formulas (for example, see \cite{Jenkins1867,
Tan00}) for the number-theoretic Jacobi symbol $\big(\frac{a}{k}\big)$ (where $k$ is odd and $\gcd(a,k)=1$).

\begin{lemma}[Eisenstein {\cite[p.~131]{Tan00}}]\label{lem:eisenstein}
Let $a$ and $k$ be odd positive integers with $\gcd(a,k)=1$. If $m \coloneq (k-1)/2,$ then
\[
\left(\frac{a}{k}\right) = (-1)^{\sum_{i=1}^{m}\left\lfloor \frac{ai}{k}\right\rfloor }.
\]
\end{lemma}

\begin{remark}
Eisenstein's lemma is often stated in the following ``counting'' form. Let
$m=(k-1)/2$, and for $1\le i\le m$ let $r_i\in\{1,2,\dots,k-1\}$ denote the least
positive residue of $ai \pmod{k}$. Then
\[
\left(\frac{a}{k}\right)=(-1)^{\,\#\{\,1\le i\le m:\ r_i>m\,\}}.
\]
Equivalently, writing the even residue set as
$E_k=\{2,4,\dots,k-1\}=\{2i:1\le i\le m\}$, one has
\[
\left(\frac{a}{k}\right)=(-1)^{\sum_{i=1}^{m}\left\lfloor \frac{2ai}{k}\right\rfloor},
\]
since $ai=q_i k+r_i$ implies
$\left\lfloor \frac{2ai}{k}\right\rfloor=2q_i+\mathbf 1_{r_i>m}$, so the parity of
$\sum_{i=1}^m \big\lfloor \tfrac{2ai}{k}\big\rfloor$ is exactly the number of residues
$r_i$ lying in the ``upper half'' $\{m+1,\dots,k-1\}$. We use the floor-sum
formulation stated in Lemma~\ref{lem:eisenstein}, which is an equivalent form of Eisenstein's lemma
for odd $a$ and $k$.
\end{remark}

The next lemma gives another formulation for these Jacobi symbols,
concerning the location of least residues.

\begin{lemma}[Gauss--Schering {\cite{Jenkins1867}}]\label{lem:gauss_schering}
Let $a$ and $k$ be positive integers with $k$ odd and $\gcd(a,k) = 1$.
Let $r_k(s)$ denote the least nonnegative residue of $s$ modulo $k$ and let
\[
    H_k=\{\,j: 1\le j\le (k-1)/2\,\} \quad \text{ and }\quad
    W_{a,k}=\{\,r_k(a j): j\in H_k\,\}.
\]
If we define
\[
m(a,k) \coloneq \big|\{\,w\in W_{a,k} : w\notin H_k\,\}\big|,
\]
then we have
\[
\Big(\frac{a}{k}\Big) = (-1)^{m(a,k)}.
\]
\end{lemma}

Finally, we recall the simple well-known closed formula for $\Big(\frac{2}{k}\Big).$

\begin{lemma}[Supplementary law for $\big(\frac{2}{k}\big)$ {\cite[p.~131]{Tan00}}]\label{lem:supp2}
Let $k$ be an odd positive integer. Then
\[
\Big(\frac{2}{k}\Big)=(-1)^{(k^2-1)/8}=
\begin{cases}
+1,& k\equiv \pm 1\pmod 8,\\
-1,& k\equiv \pm 3\pmod 8.
\end{cases}
\]
\end{lemma}

Lemmas~\ref{lem:eisenstein}, \ref{lem:gauss_schering}, and \ref{lem:supp2}
reduce the conjecture to the study of
a combinatorial sum $F_k(a),$ which we now define.
If $k \ge 3$ is an odd integer, then let $m \coloneq (k-1)/2$. For any integer $a$, we define the function
\begin{equation}
F_k(a) \coloneq \sum_{i=1}^{m} \left\lfloor \frac{ai+m}{k} \right\rfloor.
\label{eq:F}
\end{equation}
The lemmas above imply the following claim.

\begin{lemma}\label{lem:Fk}
Let $k \ge 3$ be odd.
For any positive integer $a$ with $\gcd(a,k)=1$, we have
\[
F_k(a)\equiv \begin{cases} 0\pmod 2 &{\text {if $a$ is odd}},\\
\lfloor \frac{k+1}{4}\rfloor\pmod 2 &{\text {if $a$ is even}}.
\end{cases}
\]
\end{lemma}

\begin{proof}
Let $k\ge 3$ be odd and set $m=(k-1)/2$.
For each $1\le i\le m$, write
\[
ai=kq_i+r_i,\qquad q_i=\left\lfloor \frac{ai}{k}\right\rfloor,\qquad 1\le r_i\le k-1,
\]
where $r_i=r_k(ai)$ is the least positive residue (note $r_i\neq 0$ since $\gcd(a,k)=1$ and $1\le i<k$).
Then we have
\[
\left\lfloor \frac{ai+m}{k}\right\rfloor
=\left\lfloor \frac{kq_i+r_i+m}{k}\right\rfloor
=q_i+\left\lfloor \frac{r_i+m}{k}\right\rfloor
=q_i+\mathbf{1}_{\{r_i>m\}},
\]
because $0<r_i<k$ and $r_i+m\ge k$ if and only if $r_i\ge m+1$. Summing the preceding identity over $i=1,\dots,m$ gives
\begin{equation}\label{eq:Fa_split}
F_k(a)=\sum_{i=1}^{m}\left\lfloor \frac{ai}{k}\right\rfloor+\#\{\,1\le i\le m:\ r_k(ai)>m\,\}.
\end{equation}
If we let
\[
S_k(a) \coloneq \sum_{i=1}^{m}\left\lfloor \frac{ai}{k}\right\rfloor,
\]
then the second term in \eqref{eq:Fa_split} is exactly $m(a,k)$ from
Lemma~\ref{lem:gauss_schering} (it counts those residues $r_k(ai)$ which lie outside $H_k=\{1,\dots,m\}$).
Hence, we have
\begin{equation}\label{eq:Fa_S_mu}
F_k(a)=S_k(a)+m(a,k).
\end{equation}

\medskip
\noindent\emph{Case where $a$ is odd.}
By Lemma~\ref{lem:eisenstein}, $\big(\frac{a}{k}\big)=(-1)^{S_k(a)}$, and by Lemma~\ref{lem:gauss_schering},
$\big(\frac{a}{k}\big)=(-1)^{m(a,k)}$.
Therefore, we have $S_k(a)\equiv m(a,k)\pmod 2$, and from \eqref{eq:Fa_S_mu} we obtain
\[
F_k(a)\equiv S_k(a)+m(a,k)\equiv 0\pmod 2.
\]

\medskip
\noindent\emph{Case where $a$ is even.}
Then $a+k$ is odd and $\big(\frac{a+k}{k}\big)=\big(\frac{a}{k}\big)$.
Applying Lemma~\ref{lem:eisenstein} to $a+k$ yields
\[
\Big(\frac{a}{k}\Big)=\Big(\frac{a+k}{k}\Big)
=(-1)^{\sum_{i=1}^{m}\left\lfloor \frac{(a+k)i}{k}\right\rfloor }.
\]
But for each $i$ we have
\[
\left\lfloor \frac{(a+k)i}{k}\right\rfloor
=\left\lfloor \frac{ai}{k}+i\right\rfloor
=\left\lfloor \frac{ai}{k}\right\rfloor+i,
\]
and so we obtain
\[
\sum_{i=1}^{m}\left\lfloor \frac{(a+k)i}{k}\right\rfloor
=S_k(a)+\sum_{i=1}^{m} i.
\]
On the other hand, Lemma~\ref{lem:gauss_schering} gives $\big(\frac{a}{k}\big)=(-1)^{m(a,k)}$.
Therefore, we have
\[
S_k(a)+\sum_{i=1}^{m} i \equiv m(a,k)\pmod 2.
\]
Using \eqref{eq:Fa_S_mu}, we conclude
\[
F_k(a)=S_k(a)+m(a,k)\equiv \sum_{i=1}^{m} i \pmod 2.
\]
Since $\sum_{i=1}^{m} i = m(m+1)/2 = (k^2-1)/8$, we have
\begin{equation}\label{eq:Fa_even}
F_k(a)\equiv \frac{k^2-1}{8}\pmod 2.
\end{equation}
Finally, because $k$ is odd, $k\equiv 1,3,5,7\pmod 8$.
By the supplementary law (Lemma~\ref{lem:supp2}), the exponent
$\frac{k^2-1}{8}$ satisfies $\big(\frac{2}{k}\big)=(-1)^{(k^2-1)/8}$, so its parity
is $0$ when $k\equiv\pm1\pmod 8$ and $1$ when $k\equiv\pm3\pmod 8$.
A direct check gives
\[
\frac{k^2-1}{8}\equiv
\begin{cases}
0 & (k\equiv 1,7\!\!\!\pmod 8),\\
1 & (k\equiv 3,5\!\!\!\pmod 8)
\end{cases}
\qquad\text{and}\qquad
\left\lfloor \frac{k+1}{4}\right\rfloor\equiv
\begin{cases}
0 & (k\equiv 1,7\!\!\!\pmod 8),\\
1 & (k\equiv 3,5\!\!\!\pmod 8)
\end{cases}
\]
modulo $2$.
Thus $\frac{k^2-1}{8}\equiv \left\lfloor\frac{k+1}{4}\right\rfloor\pmod 2$, and \eqref{eq:Fa_even} implies
\[
F_k(a)\equiv \left\lfloor \frac{k+1}{4}\right\rfloor \pmod 2.
\]
This completes the proof.
\end{proof}

\subsection{Proof of Theorem~\ref{thm:conj}}

Thanks to Lemma~\ref{lem:Fk}, Theorem~\ref{thm:conj} follows from the following combinatorial identity.

\begin{lemma}[Formula for $N_k(n)$]\label{lem:N_as_floor_diff}
If  $k$ is odd, $m=(k-1)/2$, and $n$ satisfies $\gcd(n,k)=1$, then
we have
\[
N_k(n)=F_k(n+1)-F_k(n).
\]
\end{lemma}

\begin{proof}
Fix $b\in\{1,\dots,m\}$ and write
\[
nb=qk+r,\qquad 1\le r\le k-1,
\]
(which is possible since $\gcd(n,k)=1$).
Then we have
\[
\left\lfloor\frac{(n+1)b+m}{k}\right\rfloor-\left\lfloor\frac{nb+m}{k}\right\rfloor
=\left\lfloor\frac{r+m+b}{k}\right\rfloor-\left\lfloor\frac{r+m}{k}\right\rfloor.
\]
Set $x \coloneq r+m$. Since $1\le r\le k-1$ and $m=(k-1)/2$, we have
\[
m+1\le x\le (k-1)+m=\frac{3k-3}{2},
\]
(with equality possible at the right endpoint), and also $1\le b\le m<k$.
In particular, since $x\le \frac{3k-3}{2}$ and $b\le m=\frac{k-1}{2}$, we have
\[
x+b \le \frac{3k-3}{2}+\frac{k-1}{2}=2k-2<2k,
\]
and clearly $x < x+b < x+k$. Consequently, we have
\[
\left\lfloor\frac{x+b}{k}\right\rfloor-\left\lfloor\frac{x}{k}\right\rfloor\in\{0,1\},
\]
and this difference equals $1$ if and only if $x<k\le x+b$.
Since $x=r+m$, the condition $x<k$ is equivalent to $r\le m$, and the condition
$k\le x+b$ is equivalent to
\[
r \ge k-(m+b)=m+1-b.
\]
Therefore, we have
\[
\left\lfloor\frac{(n+1)b+m}{k}\right\rfloor-\left\lfloor\frac{nb+m}{k}\right\rfloor
=
\mathbf{1}_{\{\,m+1-b\le r\le m\,\}}.
\]

Now, the congruence condition $b_2\equiv nb\pmod{k}$ with $1\le b_2\le m$ is solvable
if and only if $r\le m$, in which case the solution is unique and given by $b_2=r$.
Under this identification, the inequality $b+b_2\ge m+1$ becomes
\[
b+r\ge m+1\quad\Longleftrightarrow\quad r\ge m+1-b.
\]
Thus, for each fixed $b\in\{1,\dots,m\}$, the indicator that there exists a (unique) $b_2$
making $(b,b_2)$ contribute to $N_k(n)$ is exactly
\[
\left\lfloor\frac{(n+1)b+m}{k}\right\rfloor-\left\lfloor\frac{nb+m}{k}\right\rfloor.
\]
Summing over $b=1,\dots,m$ gives the desired conclusion
\[
    N_k(n)
    = \sum_{b=1}^m\left(\left\lfloor\frac{(n+1)b+m}{k}\right\rfloor-\left\lfloor\frac{nb+m}{k}\right\rfloor\right)
    = F_k(n+1)-F_k(n). \qedhere
\]
\end{proof}
\begin{remark}
  \label{rem:gcd_not_needed}
Lemma~\ref{lem:N_as_floor_diff} also turns out to be true even
  without the condition $\gcd(n,k) = 1$.
\end{remark}

\begin{proof}[Proof of Theorem~\ref{thm:conj}]
Fix an odd integer $k$ and set $m=(k-1)/2$. Let $n$ satisfy
$\gcd(n,k)=\gcd(n+1,k)=1$ as in Conjecture~\ref{conj}.
By Lemma~\ref{lem:N_as_floor_diff}, we have
\[
N_k(n)=F_k(n+1)-F_k(n).
\]
Hence, it follows that
\[
N_k(n)\equiv F_k(n+1)+F_k(n)\pmod{2}.
\]

Exactly one of $n$ and $n+1$ is even; let $e\in\{n,n+1\}$ be the even one
and let $o\in\{n,n+1\}$ be the odd one. By assumption,
$\gcd(e,k)=\gcd(o,k)=1$, so Lemma~\ref{lem:Fk} applies to both $e$ and $o$.
Since $o$ is odd, Lemma~\ref{lem:Fk} gives $F_k(o)\equiv 0\pmod{2}$, and therefore
\[
N_k(n)\equiv F_k(e)\pmod{2}.
\]
Since $e$ is even, Lemma~\ref{lem:Fk} also gives
\[
F_k(e)\equiv \left\lfloor\frac{k+1}{4}\right\rfloor \pmod{2}.
\]
Combining the last two congruences yields
\[
N_k(n)\equiv \left\lfloor\frac{k+1}{4}\right\rfloor \pmod{2},
\]
which is exactly Conjecture~\ref{conj}, and  Theorem~\ref{thm:conj} holds.
\end{proof}

\begin{remark}
    After observing the relationship with Jacobi symbols,
    we can also provide a more concise description for $n_k(\mu)$ defined in \eqref{eq:n_k}.
    Let $d_i = \gcd (k, m_i)$ for $i = 1, \ldots, n$, where $\mu = (m_1,\ldots, m_n)$. Then we have
    \begin{equation}
    \label{eq:n_k_Jacobi}
    n_k(\mu)=
    \# \left \{ i \ : \
    \Big(\frac{2}{d_i}\Big) \neq \Big(\frac{2}{k}\Big) \right \}.
    \end{equation}
    To see this, observe that for each prime $q_i$ in the factorization of $k$ and any integer $m$,
    \[ \min\{\nu_{q_i}(m), \nu_{q_i}(k)\} = \nu_{q_i}(\gcd(m, k)). \]
    Therefore, for $d = \gcd(m, k)$, it follows that
    \[ \nu_{\mathcal Q}(m) = \sum_{i=1}^t \nu_{q_i}(d). \]
    Additionally, since $\Big(\frac{2}{p_i}\Big) = 1$ and $\Big(\frac{2}{q_i}\Big) = -1$, we have
    \[ \Big(\frac{2}{d}\Big) = \prod_{i=1}^t \Big(\frac{2}{q_i}\Big)^{\nu_{q_i}(d)} = (-1)^{\sum_{i=1}^t \nu_{q_i}(d)} = (-1)^{\nu_{\mathcal Q}(m)}. \]
    Similarly, we have
    \[ \Big(\frac{2}{k}\Big) = (-1)^{\nu_{\mathcal Q}(k)}. \]
    Hence, the condition
    $\nu_{\mathcal Q}(m_i)\not\equiv \nu_{\mathcal Q}(k) \pmod{2}$ in the original definition of $n_k(\mu)$
    is equivalent to
    \[ \Big(\frac{2}{d_i}\Big) \neq \Big(\frac{2}{k}\Big). \]
\end{remark}


\section*{Appendix: AxiomProver and Lean verification}\label{sec:AI}
AxiomProver, an AI system under development for formal mathematical proof,
was tested using Conjecture~\ref{conj} as an early case target.
The system produced the key reformulation and proof strategy underlying the proof presented in this note.
We stress at the outset that the object formalized in Lean is the
combinatorial identity of Lemma~\ref{lem:N_as_floor_diff} (equivalently, the
number-theoretic content of Conjecture~\ref{conj}); the geometric results on
$k$-differentials, Theorems~\ref{thm:zero} and~\ref{thm:one}, were not formalized.
Since a formalized proof reduces the question of correctness to the question of
whether the formalized statement faithfully captures the intended one, we display
and explain that statement below so that the reader can check it against
Conjecture~\ref{conj} and Lemma~\ref{lem:N_as_floor_diff}.

For transparency, we describe the provenance.
In the discovery phase, we provided AxiomProver only the conjecture statement,
with no references to Chen–Gendron, geometric context, or method hints (like Jacobi symbols).
It returned the Jacobi-symbol reformulation and a reduction to the explicit
floor-sum in Lemma~\ref{lem:N_as_floor_diff}.
This reduction is the key elusive step that had not been noticed earlier.

\subsection*{Process}
We asked AxiomProver to verify Lemma~\ref{lem:N_as_floor_diff} in Lean (see \cite{Lean,Mathlib2020}).
The formal proofs provided in this work were developed and verified using Lean \textbf{4.26.0}.
Compatibility with earlier or later versions is not guaranteed due to the evolving nature of the Lean 4 compiler and its core libraries.
The relevant files are all posted in the following repository:
\begin{center}
  \url{https://github.com/AxiomMath/parity-differential}
\end{center}
The input files were
\begin{itemize}
  \item a \texttt{task.md} containing the natural-language problem statement; and
  \item a configuration file \texttt{.environment} that contains the single line
  \begin{quote}
    \slshape
    lean-4.26.0
  \end{quote}
  which specifies the version of Lean that AxiomProver should use.
\end{itemize}
Given these two input files, AxiomProver autonomously provided the following output files:
\begin{itemize}
  \item \texttt{problem.lean}, a Lean 4.26.0 formalization of the problem statement; and
  \item \texttt{solution.lean}, a complete Lean 4.26.0 formalization of the proof.
\end{itemize}
The repository also contains an ancillary file \texttt{examples.py} that
shows the verification of Lemma~\ref{lem:N_as_floor_diff} for small values of $n$.
This file was written by hand and is unrelated to the formalization process
(in particular, it was not provided to AxiomProver as part of the input).

\subsection*{The formalized statement}
For the reader's convenience we reproduce the formalized statement, so that it
can be checked directly against Lemma~\ref{lem:N_as_floor_diff}. The counting
function and the floor sum are defined as
\begin{quote}\ttfamily\small
def countingFunctionN (k n : \(\mathbb{N}\)) : \(\mathbb{N}\) :=\\
\phantom{xx}let m := (k - 1) / 2\\
\phantom{xx}Finset.card (Finset.filter\\
\phantom{xxxx}(fun p : \(\mathbb{N}\) \(\times\) \(\mathbb{N}\) => m + 1 \(\le\) p.1 + p.2 \(\wedge\) p.2 \% k = (n * p.1) \% k)\\
\phantom{xxxx}(Finset.Icc 1 m \(\times\)\({}^{\text{s}}\) Finset.Icc 1 m))\\[4pt]
def floorSumF (k a : \(\mathbb{N}\)) : \(\mathbb{N}\) :=\\
\phantom{xx}let m := (k - 1) / 2\\
\phantom{xx}\(\sum\) i \(\in\) Finset.Icc 1 m, (a * i + m) / k
\end{quote}
and the main theorem reads
\begin{quote}\ttfamily\small
theorem main\_theorem (k n : \(\mathbb{N}\))\\
\phantom{xx}(hk\_ge : 3 \(\le\) k) (hk\_odd : Odd k)\\
\phantom{xx}(\_hn\_coprime : Nat.Coprime n k) (\_hn1\_coprime : Nat.Coprime (n + 1) k) :\\
\phantom{xx}(countingFunctionN k n : \(\mathbb{Z}\)) = (floorSumF k (n + 1) : \(\mathbb{Z}\)) - (floorSumF k n : \(\mathbb{Z}\))
\end{quote}
Here \texttt{countingFunctionN k n} is exactly $N_k(n)$ and \texttt{floorSumF k a}
is exactly $F_k(a)$, so the statement is precisely
$N_k(n) = F_k(n+1) - F_k(n)$ of Lemma~\ref{lem:N_as_floor_diff}, under the
hypotheses that $k \geq 3$ is odd. Two points of Lean syntax are worth
spelling out for the non-expert reader, as neither is visible from the code alone.
First, although $F_k(a)$ is written with an explicit floor in \eqref{eq:F}, the
Lean code contains no floor symbol: the expression \texttt{(a * i + m) / k} is
division on the natural numbers $\mathbb{N}$, which in Lean and Mathlib is
\emph{defined} to be floor division, so it agrees with $\lfloor (ai+m)/k\rfloor$.
This is why the definition and its supporting lemmas carry names such as
\texttt{floorSumF} and \texttt{floor\_diff\_eq} even though $\lfloor\,\cdot\,\rfloor$
never appears literally. Second, the equality in \texttt{main\_theorem} is stated
over the integers $\mathbb{Z}$ (via the casts \texttt{(\,$\cdot$\, : \(\mathbb{Z}\))}); this is
deliberate, because subtraction on $\mathbb{N}$ is truncated (it cannot return a
negative value), and casting to $\mathbb{Z}$ ensures that $F_k(n+1) - F_k(n)$ is
the honest integer difference. Finally, the two coprimality hypotheses appear
with a leading underscore (\texttt{\_hn\_coprime}, \texttt{\_hn1\_coprime}),
which is Lean's convention for a hypothesis that is declared but never used;
see the remark below.

We verified the artifacts against the declared toolchain (Lean 4.26.0):
\texttt{solution.lean} compiles, \texttt{main\_theorem} is \texttt{sorry}-free, and
its proof depends only on the three standard axioms \texttt{propext},
\texttt{Classical.choice}, and \texttt{Quot.sound}. These checks were carried out
using AXLE (the Axiom Lean Engine), a Lean toolkit that Axiom Math has made
freely available~\cite{AXLE}, which includes its own proof verifier. AXLE
confirmed that \texttt{solution.lean} proves the statement posed in
\texttt{problem.lean}.

After AxiomProver generated a solution, the human authors wrote this paper
(without the use of AI) for human readers. Indeed, a research paper is a narrative designed to communicate ideas to humans, whereas Lean files are designed to satisfy a computer kernel.

For transparency about the process itself, we record the following. The
system was run fully autonomously, Namely, after being given the two input files
described above, it received no human feedback, hints, or intermediate
corrections while producing \texttt{problem.lean} and \texttt{solution.lean}.

\begin{remark}
In the \texttt{task.md} that we provided to AxiomProver,
we included the hypothesis $\gcd(n,k) = \gcd(n+1,k) = 1$.
However, in the resulting \texttt{solution.lean},
AxiomProver correctly identified that this hypothesis is never used
(by prefacing the hypothesis name with an underscore).
So AxiomProver's solution in fact proves the slightly improved version
of Lemma~\ref{lem:N_as_floor_diff} alluded to in Remark~\ref{rem:gcd_not_needed}.
\end{remark}

\subsection*{Further commentary}
We close with some additional perspective on these results.

\medskip
\noindent
(1) \textbf{Scope of Automation.}
    This test case involved research mathematics where the solution path was not self-contained.
    The system identified classical lemmas that enabled the Jacobi-symbolreformulation,
    and it formally verified the combinatorial identity required to complete the proof.

\smallskip
\noindent(2) \textbf{Scope of Formalization.}
    The reader will notice that we formalized the combinatorial core
    (i.e.\ Lemma~\ref{lem:N_as_floor_diff}) but not the number-theoretic reduction
    (i.e.\ Lemma~\ref{lem:Fk}) in Lean.
    This decision was driven by the dependence of Lemma~\ref{lem:Fk}
    on Lemma~\ref{lem:eisenstein} (Eisenstein) and Lemma~\ref{lem:gauss_schering} (Gauss-Schering).
    These two lemmas are already well-established in the literature,
    but would take additional work to formalize in Lean.
    We chose not to demand that the system reprove these standard facts from scratch,
    as the goal of this experiment was to test the discovery of novel proofs, not library building.

\smallskip
\noindent(3) \textbf{Implications.}
    We view this test case as a proof of principle for a research workflow:
    the mathematician poses a conjecture, and the AI can assist with the retrieval,
    reformulation, and formal verification of the components.

\bibliographystyle{plainnat}
\bibliography{biblio}

\end{document}